\documentclass[12pt]{amsart}
\usepackage[colorlinks=true,citecolor=black,linkcolor=black,urlcolor=blue]{hyperref}
\usepackage{amsmath}
\usepackage[english,  activeacute]{babel}
\usepackage[latin1]{inputenc}
\usepackage{amsthm,amsmath,amssymb}
\usepackage{graphics,graphicx}
\usepackage{enumerate}
\usepackage{cite}
\usepackage{array}
\usepackage{a4wide}
\usepackage{color, url}
\usepackage{float}
\usepackage{algorithm}
\usepackage[noend]{algpseudocode}

\usepackage{gastex}

\theoremstyle{plain}
\newtheorem{theorem}{Theorem}[section]
\newtheorem{proposition}[theorem]{Proposition}
\newtheorem{remark}[theorem]{Remark}
\newtheorem{corollary}[theorem]{Corollary}

\theoremstyle{definition}

\newcommand{\D}{{\mathcal D}}

\newcommand{\Q}{{\mathcal Q}}

\newcommand{\area}{{\texttt{area}}}

\makeatletter
\renewcommand*\env@matrix[1][*\c@MaxMatrixCols c]{%
  \hskip -\arraycolsep
  \let\@ifnextchar\new@ifnextchar
  \array{#1}}
\makeatother

\makeatletter
\def\moverlay{\mathpalette\mov@rlay}
\def\mov@rlay#1#2{\leavevmode\vtop{%
   \baselineskip\z@skip \lineskiplimit-\maxdimen
   \ialign{\hfil$\m@th#1##$\hfil\cr#2\crcr}}}
\newcommand{\charfusion}[3][\mathord]{
    #1{\ifx#1\mathop\vphantom{#2}\fi
        \mathpalette\mov@rlay{#2\cr#3}
      }
    \ifx#1\mathop\expandafter\displaylimits\fi}
\makeatother

\title[Restricted Dyck Paths]{Restricted Dyck Paths on Valleys Sequence}

\author[R. Fl\'orez]{Rigoberto Fl\'orez}\thanks{}
\address{\noindent Department of Mathematical Sciences,  The Citadel,  Charleston, SC, U.S.A.}
\email{rigo.florez@citadel.edu}

\author[T. Mansour]{Toufik Mansour}
\address{Department of Mathematics, University of Haifa,
3498838 Haifa, Israel}
\email{tmansour@univ.haifa.ac.il}

\author[J. L. Ram\'irez]{Jos\'e L. Ram\'{\i}rez}
\address{\noindent Departamento de Matem\'aticas, Universidad Nacional de Colombia, Bogot\'a,  COLOMBIA}
\email{jlramirezr@unal.edu.co}

\author[F. A. Velandia]{Fabio A. Velandia}
\address{\noindent Departamento de Matem\'aticas, Universidad Nacional de Colombia, Bogot\'a,  COLOMBIA}
\email{fvelandias@unal.edu.co}

\author[D. Villamizar]{Diego Villamizar}
\address{Department of Mathematics,
Tulane University, New Orleans, LA, U.S.A.}
\email{dvillami@tulane.edu}

\date{\today}

\begin{document}
\begin{abstract}
In this paper we study a subfamily of a classic lattice path, the \emph{Dyck paths}, called \emph{restricted $d$-Dyck} paths, in short $d$-Dyck.
A valley of a Dyck path $P$ is a local minimum of $P$;  if the difference between the heights of two consecutive valleys (from left to right) is at
least $d$, we say that $P$ is a restricted $d$-Dyck path. The \emph{area} of a Dyck path is the sum of the  absolute values of  $y$-components
of all points in the path. We find the number of peaks and the area of all paths of a given length in the set of $d$-Dyck paths. We give a   
bivariate generating function to count the number of the $d$-Dyck paths with respect to the the semi-length and number of peaks. After that,    
we analyze  in detail the case $d=-1$. Among other things, we give both, the generating function and a recursive relation for the total area.
\end{abstract}

\maketitle

\section{Introduction}

A classic concept, the \emph{Dyck paths}, has been widely studied. Recently, a subfamily of these paths, non-decreasing Dyck paths, has received certain
level of interest  due to the good behavior of its recursive relations and generating functions. In this paper we keep
studying a generalization of  the non-decreasing Dyck paths. Other generalizations of non-drecreasing Dyck paths have been given for Motzkin paths and for \L{}ukasiewicz paths  \cite{RigoJoseLuis,RigoJoseLuisRational}.

We now recall, to avoid ambiguities, some important definitions that we need in this paper. A Dyck path is a lattice path in the first quadrant of the $xy$-plane that starts at the origin, ends on the $x$-axis, and consists of (the same number of) North-East steps  $U:=(1,1)$ and South-East steps $D:=(1,-1)$. The \emph{semi-length} of a path is the total number of $U$'s that the path has. A \emph{valley} (\emph{peak}) is a subpath of the form $DU$
($UD$) and the \emph{valley vertex} of  $DU$ is the lowest point (a local minimum) of $DU$. Following \cite{FloRamVelVilPolyominoes, FloRamVelVil} we define the \emph{valley vertices} of a Dyck path $P$ as the vector
$\nu=(\nu_1, \nu_2, \dots, \nu_k)$ formed by all $y$-coordinates (listed from left to right) of all valley vertices of $P$. For further recent work about different combinatorial aspects of Dyck paths, see for instance \cite{Baril, Baril2,  Blecher, RigoJoseLuisSym, Manes, Sergi}.

For a fixed $d \in \mathbb{Z}$,
a Dyck path $P$ is called  \emph{restricted $d$-Dyck} or  \emph{$d$-Dyck}  (for simplicity),  if either $P$  has at most one valley, or  if its valley vertex
vector $\nu$ satisfies that $\nu_{i+1}-\nu_i\geq d$, where   $1\le i <k$.  The set of all $d$-Dyck paths is denoted by $\D_{d}$, the set of all
$d$-Dyck paths of semi-length $n$ is denoted $\D_{d}(n)$, and the cardinality of $\D_{d}(n)$ is  denoted by $r_d(n)$.

The first well-known example of these paths is the set of $0$-Dyck paths; in the literature,
\cite{barcucci,CZ,CZ2,DeutscheProdinger, ElizaldeFlorezRamirez,FlorezJose}, this family is known as non-decreasing Dyck  paths.  The whole family of Dyck paths can
be seen as a limit of $d$-Dyck and it occurs occurs when $d\to -\infty$. Another example is from Figure \ref{Example}; we observe that $\nu=(0,1,0,3,4,3,2)$ and that
$\nu_{i+1}-\nu_i\ge -1$, for $i=1, \dots, 6$, so the figure depicts a $(-1)$-Dyck path of length 28 (or semi-length 14).

\begin{figure} [htbp]
\begin{center}
\includegraphics[scale=0.9]{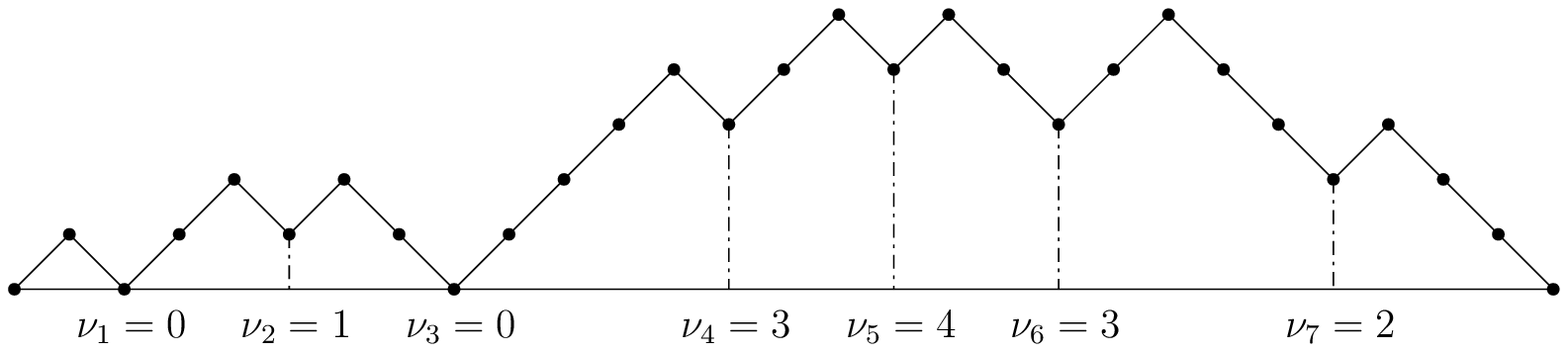}
\end{center}
\caption{A $(-1)$-Dyck path of length 28.} \label{Example}
\end{figure}

The recurrence relations and/or the generating functions for $d$-Dyck when $d\ge 0$ have different behavior than the case  $d< 0$. For example, the generating
functions  for known aspects, in $d$-Dyck when $d\ge 0$, are all rational (see \cite{barcucci,CZ,CZ2,ElizaldeFlorezRamirez,FlorezJose,FloRamVelVilPolyominoes, FloRamVelVil}).
However, the aspects that we analyze in this paper, when $d< 0$, give that the generating functions are all algebraic (non-rational).  
In this paper we give a bivariate generating function to count the number of paths    
in $\D_{d}(n)$, for $d\le 0$,  with respect to the number of peaks and semi-length. We also give a relationship  between the total number of   
$d$-Dyck paths and the Catalan numbers. Additionally, we give an explicit  symbolic expression for the generating function with respect to the   
semi-length.  For the particular case $d=-1$ we give a combinatorial expression and a recursive relation for the total number of paths.  
We also analyze the asymptotic behavior for the sequence $r_{-1}(n)$. It would be very interesting if we can understand better the behavior of  $d$-Dyck paths for $d<-1$. 

The \emph{area} of a Dyck path is the sum of the  absolute values of  $y$-components of all points in the path. That is,  the area  of a  Dyck path
 corresponds to the surface area under the paths and above of the $x$-axis.  For example, the path $P$ in Figure \ref{Example} satisfies that  $\area(P)=70$. We use generating functions and recursive relations to analyze the distribution of the area of all paths
in  $\D_{-1}(n)$.

 \section{Number of $d$-Dyck paths and Peaks Statistic}
Given a family of lattice paths, a classic question is how many lattice paths are there of certain length, and a second classic question is
how many peaks are there depending on the length of the path. These questions have been completely answered, for instance, for Dyck paths  
\cite{Deutschenumeration},  $d$-Dyck paths for $d\ge 0$ \cite{barcucci, FloRamVelVil},  and Motzkin paths \cite{sapounakis} among others.  
In this section we give a bivariate generating  function to  enumerate the peaks and semi-length of the $d$-Dyck paths for $d< 0$. 

We now give some notation needed for this paper, including the parameters needed for the generating function in this section. The   
\emph{level} of a valley is the $y$-component of its valley vertex.  We recall that the set of all $d$-Dyck paths is denoted by $\D_{d}$; the set   
of all $d$-Dyck paths of semi-length $n$ is denoted $\D_{d}(n)$, and the cardinality of $\D_{d}(n)$ is  denoted by $r_d(n)$. Given a 
$d$-Dyck path $P$, we denote the semi-length of $P$ by $\ell(P)$ and denote the number of peaks of $P$ by $\rho(P)$.  So,  the 
bivariate generating function to count the number of paths and peaks of $d$-Dyck paths is defined by
 $$L_d(x, y):=\sum_{P\in \D_d}x^{\ell(P)}y^{\rho(P)}.$$

 \subsection{Some facts known when $d\geq 0$.}
These results can be found in \cite{FloRamVelVil}.
\begin{itemize}
\item
If $d\geq 0$, then the generating function $F_d(x, y)$ is given by 
  $$L_d(x, y)=1  + \frac{xy(1-2x+x^2+xy-x^{d+1}y)}{(1-x)(1-2x+x^2-x^{d+1}y)}.$$
\item If  $d\geq 1$,
$$r_d(n)=\sum_{k=0}^{\lfloor\frac{n+d-2}{d}\rfloor}\binom{n-(d-1)(k-1)}{2k}.$$
\item If $n> d$, then   we have the recursive relation
$$r_d(n)=2r_d(n-1)-r_d(n-2)+r_d(n-d-1),$$
with the initial values $r_d(n)=\binom n2 +1$, for $0\leq n \leq d$.

\item Let $p_d(n,k)$ be the number of $d$-Dyck paths of semi-length $n$, having exactly $k$ peaks.
If $d\geq 0$, then
$$p_d(n,k)=\binom{n+k-d(k-2)-2}{2(k-1)}.$$
For the whole set of Dyck paths, the number $p_{-\infty}(n,k)$, is given by the Narayana numbers $N(n,k)=\frac{1}{n}\binom nk \binom{n}{k-1}$.

\end{itemize}

\subsection{Peaks statistic for $d$ a negative integer}  For the remaining part of the paper we consider only the case $d<0$ and use  $e$ to
denote $|d|$.  

\begin{theorem}\label{mainGF} If $d$ is a negative integer and $e:=|d|$, then
the generating function $L_e(x, y)$ satisfies the functional equation
\begin{align}
  L_e(x,y)= xy+xL_e(x,y)+xS_e(x,y)L_e(x,y), \label{eqSys1}
\end{align}
where $S_e(x)$ satisfies the algebraic equation
\[
(1-xS_e(x,y))^e(y+(1-y)xS_e(x,y))-S_e(x,y)(1-xS_e(x,y))^{e+1}-\frac{x^{e+2}y}{1-x}S_e(x,y)=0.
\]
\end{theorem}

\begin{proof} We start this proof introducing some needed notation. The set $\Q_{d,i} \subseteq \D_{d}$ denotes the family of nonempty paths
where the last valley is at level $i$. We consider the generating function
$$Q_i^{(e)}(x,y):=\sum_{P\in \Q_{d,i}}x^{\ell(P)}y^{\rho(P)}.$$
It is convenient to consider the sum over the $Q^{(e)}_i(x,y)$. We also consider the generating function, with respect  to the lengths and peaks, that counts  the $d$-Dyck paths  that have either no valleys or the last valley is at level less than $e$. That is, 
\begin{align}\label{defSe}
S_e(x,y)=\frac{y}{1-x}+\sum _{j=0}^{e-1}Q^{(e)}_j(x,y).
\end{align}
 A path  $P$ can be uniquely decomposed as either
$UD, UTD$,   or $UQDT$ (by considering the first return decomposition), where $T\in\D_{d}$ and $Q$ is either a path without valleys or is a path in $\cup_{i=0}^{e-1}\Q_{d,i}$ (see Figure \ref{Fig3}, for a graphical representation of this decomposition). Notice that the decomposition $UQDT$ ensures that the condition $\nu_{i+1}-\nu_i\geq d$ holds for all $i\geq 1$.

\begin{figure} [htbp]
\begin{center}
\includegraphics[scale=0.8]{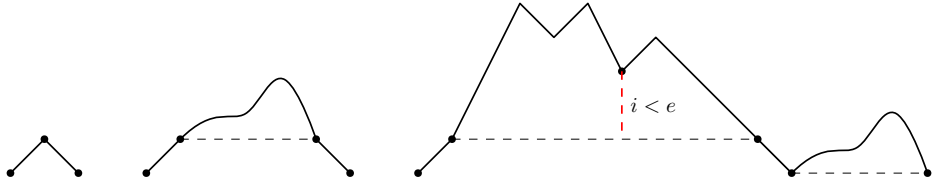}
\end{center}
\caption{Decomposition of a $d$-Dyck path.} \label{Fig3}
\end{figure}

 From the symbolic method we obtain the functional equation
$$  L_e(x,y)= xy+xL_e(x,y)+xS_e(x,y)L_e(x,y).$$
Now we are going to obtain a system of equations  for the generating functions $Q_i(x,y)$. Let $Q$ be a path in the set $\Q_{d,i}$. If $i=0$, then the path $Q$ can be decomposed uniquely as either
$UQ'D\Delta$   or $UQ'DR$, where $\Delta$ is a pyramid, $R$ is a path in $\Q_{d,0}$, and  $Q'$ is either a path without valleys or $Q'\in \cup_{i=0}^{e-1}\Q_{d,i}$.
Therefore, we have the functional equation
$$Q_0^{(e)}(x,y)=xS_e(x,y)\frac{xy}{1-x}+xS_e(x,y)Q_0^{(e)}(x,y).$$
For $i>0$, any path  $Q$ can be decomposed uniquely in one of  these two forms $UR_1D$  or $UQDR_2$,
where $R_1\in \Q_{d,i-1}, R_2 \in \Q_{d,i}$,  and $Q$ is either a path without valleys or $Q\in \cup_{i=0}^{e-1}\Q_{d,i}$.
So, we have the functional equation
$$Q_i^{(e)}(x,y)=xQ_{i-1}^{(e)}(x,y)+xS_e(x)Q_i^{(e)}(x,y).$$

Summarizing the above discussion, we obtain the system of equations:
\begin{align}
\label{eqSys2}
\begin{cases}
Q_0^{(e)}(x,y)&=xS_e(x,y)\frac{xy}{1-x}+xS_e(x,y)Q_0^{(e)}(x,y)\\
    Q_1^{(e)}(x,y) &=xQ_0^{(e)}(x,y)+xS_e(x,y)Q_1^{(e)}(x,y)\\
    &\ \vdots\\
    Q_i^{(e)}(x,y)&=xQ_{i-1}^{(e)}(x,y)+xS_e(x,y)Q_i^{(e)}(x,y)\\
    &\ \vdots \\
    Q_{e-1}^{(e)}(x,y)&=xQ_{e-2}^{(e)}(x,y)+xS_e(x,y)Q_{e-1}^{(e)}(x,y).
    \end{cases}
    \end{align}
Summing the equations in \eqref{eqSys2}, we obtain that
$$\sum_{j=0}^{e-1}Q_j^{(e)}(x,y)=xS_e(x,y)\left(\sum_{j=0}^{e-1}Q_j^{(e)}(x,y)+\frac{xy}{1-x}\right)+x\sum_{j=0}^{e-2}Q_j^{(e)}(x,y).$$
From this and \eqref{defSe} we have
\begin{multline}\label{eqSex}
    S_e(x,y)-\frac{y}{1-x}=x\left (S_e(x,y)-\frac{y}{1-x}-Q_{e-1}^{(e)}(x,y)\right)\\+xS_e(x,y)\left(S_e(x,y)-\frac{y}{1-x}\right )+\frac{x^2y}{1-x}S_e(x,y).
\end{multline}
Now solving this in previous equation for $S_e(x,y)$ we have
\begin{align}\label{FunEcS}
S_e(x,y)=\frac{1-x+xy-\sqrt{1-2x+x^2-2xy-2x^2y+x^2y^2+4x^2Q_{e-1}^{(e)}(x,y)}}{2x}.
\end{align}
Notice that all of the $Q_i^{(e)}(x,y)$, with $i\geq 0$, can be expressed as
\begin{align}\label{FunEcQ}
Q_i^{(e)}(x,y)=\frac{x^{i+2}yS_e(x,y)}{(1-x)(1-xS_e(x,y))^{i+1}}.
\end{align}
Substituting \eqref{FunEcQ} into \eqref{eqSex} we obtain the desired functional equation.
\end{proof}

We observe that substituting   \eqref{FunEcS} into  \eqref{eqSys1}, we have
\begin{eqnarray*}
L_e(x,y)&=&\dfrac{xy}{1-x-xS_e(x,y)}\\
&=&\dfrac{xy}{1-x-\dfrac{1-x+xy-\sqrt{1-2x+x^2-2xy-2x^2y+x^2y^2+4x^2Q_{e-1}^{(e)}(x,y)}}{2}}.
\end{eqnarray*}
From the  combinatorial description of $Q^{(e)}_{e-1}(x,y)$, we obtain  $Q_{e-1}^{(e)}(x,y)\longrightarrow  0$, as $e\longrightarrow \infty$. Therefore,
$$\lim_{e\to \infty} L_e(x,y)=\lim_{e\to \infty}\frac{xy}{1-x-xS_e(x,y)}=\frac{1-x-xy-\sqrt{1 - 2 x + x^2 - 2 x y - 2 x^2 y + x^2 y^2}}{2x}.$$
This last generating function is the distribution of the Narayan sequence. This corroborates that the restricted $(-\infty)$-Dyck paths coincides
with the non-empty Dyck paths.

 \begin{theorem}
 If $1\leq k \leq |d|+3$, then the $k$-th coefficient of the generating function  $L_e(x,1)$ coincides with the  Catalan number $C_k$.
 \end{theorem}
\begin{proof}  We first observe that the shortest Dyck path that contains a forbidden sequence of valleys is $P=U^{e+2}DUD^{e+2}UD$ (clearly,
$\ell (P)=e+4$) with $e=|d|$. Therefore,  if $d<0$, then $r_d(n)=C_n$ for $n=1, 2, \dots, |d|+3$.
\end{proof}
The first few values for the sequence $r_d(n)$, for $d\in \{-1, -2, -3, -4\}$ are
\begin{align*}
\{r_{-1}(n)\}_{n\geq 1}&=\{\textbf{1}, \, \textbf{2}, \, \textbf{5}, \, \textbf{14}, \, 41, \, 123, \, 375, \, 1157, \, 3603, \dots \}, \\
\{r_{-2}(n)\}_{n\geq 1}&=\{\textbf{1}, \, \textbf{2}, \, \textbf{5}, \, \textbf{14}, \, \textbf{42}, \, 131, \, 419, \, 1365, \, 4511, \dots \}, \\
\{r_{-3}(n)\}_{n\geq 1}&=\{\textbf{1}, \, \textbf{2}, \,  \textbf{5}, \,  \textbf{14}, \,  \textbf{42}, \,  \textbf{132}, \,  428, \,  1419, \,  4785, \dots \}, \\
\{r_{-4}(n)\}_{n\geq 1}&=\{\textbf{1}, \, \textbf{2}, \, \textbf{5},\,  \textbf{14}, \, \textbf{42},\,  \textbf{132}, \, \textbf{429}, \, 1429, \, 4850, \dots \}.
\end{align*}
For example,  there are $41$ $(-1)$-Dyck paths out of the $42$ Dyck paths of length $10$. The Figure \ref{DyckF}, depicts the only Dyck path of length 10 that is not a $(-1)$-Dyck path.

\begin{figure} [htbp]
\begin{center}
\includegraphics[scale=1.3]{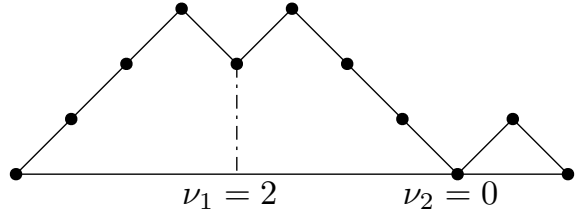}
\end{center}
\caption{The only Dyck path of length 10 that is not a $(-1)$-Dyck path.} \label{DyckF}
\end{figure}

Recall that $d$ is a negative integer and that $e:=|d|$. Then by Theorem \ref{mainGF}, we have
\begin{align*}
(L_e(x,y)+y)^e&\left(xL^2_e(x,y)+(xy+x-1)L_e(x,y)+xy\right)\\
&-\frac{x}{1-x}((1-x)L_e(x,y)-xy)(L_e(x,y))^{e+1}=0.
\end{align*}
This implies that 
\begin{multline*}
\sum_{j=2}^{e+1}x\binom{e}{j-2}y^{e+2-j}(L_e(x,y))^j+\sum_{j=1}^{e+1}(xy+x-1)\binom{e}{j-1}y^{e+1-j}(L_e(x,y))^j\\
-\sum_{j=0}^ex\binom{e}{j}y^{e+1-j}(L_e(x,y))^j+\frac{x^2y}{1-x}(L_e(x,y))^{e+1}=0.
\end{multline*}
Hence, by taking $y=1$, we have
$$L_e(x,1)=Z\left(a_0+\sum_{j=2}^{e+1}a_j(x)(L_e(x,1))^j\right),$$
where $Z=1$, and
\begin{align*}
a_0&=\frac{x}{1-(e+2)x},\\
a_j&=\frac{1}{1-(e+2)x}\left(x\binom{e+2}{j}-\binom{e}{j-1}\right),\quad j=2,3,\ldots,e,\\
a_{e+1}&=\frac{(e+2)x(1-x)-1+x(1+x)}{(1-x)(1-(e+2)x)}.
\end{align*}
Hence, by the Lagrange inversion formula, we expand the generating function $L_e(x,1)$ as a power series in $Z$ to obtain 
\begin{align*}
L_e(x,1)&=\sum_{n\geq1}\frac{[Z^{n-1}]}{n}
\sum_{i_0+i_2+i_3+\cdots+i_{e+1}=n}\frac{n!}{i_0!i_2!\cdots i_{e+1}!}a_0^{i_0}Z^{2i_2+\cdots+(e+1)i_{e+1}}\prod_{j=2}^{e+1}a_j^{i_j},
\end{align*}
that leads to the following result.
\begin{theorem}\label{thLfe}
We have 
\begin{align*}
L_e(x,1)&=\sum_{n\geq1}\frac{\sum_{2i_2+\cdots+(e+1)i_{e+1}=n-1}\binom{n}{i_2,\ldots,i_{e+1}}x^{n-i_2-\cdots-i_{e+1}}t^{i_{e+1}}\prod_{j=2}^{e}\left(x\binom{e+2}{j}-\binom{e}{j-1}\right)^{i_j}}{n(1-(e+2)x)^n},
\end{align*}
where 
\[ \binom{n}{i_2,\ldots,i_{e+1}}=\frac{n!}{i_2!\cdots i_{e+1}!(n-i_2-\cdots-i_{e+1})!} \text{ and } t=\frac{(e+2)x(1-x)-1+x(1+x)}{1-x}.\]
\end{theorem}
For example, Theorem \ref{thLfe} with $e=2$ gives
\[ L_2(x,1)=\sum_{n\geq1}\frac{\sum_{2i_2+3i_3=n-1}\binom{n}{i_2,i_3}x^{n-i_2-i_3} (6x-2)^{i_2}(\frac{-3x^2+5x-1}{1-x})^{i_3}}{n(1-4x)^n}.\]
Thus, 
\begin{multline*}L_2(x,1)=\frac{x}{1-4x}+\frac{x^2(6x-2)}{(1-4x)^3}+\frac{x^3t}{(1-4x)^4}+
\frac{2x^3(6x-2)^2}{(1-4x)^5} +\frac{5x^4t(6x-2)}{(1-4x)^6}\\
+\frac{5x^4(6x-2)^3+3x^5t^2}{(1-4x)^7}+\frac{21x^5 (6 x - 2)^2 t}{(1 - 4 x)^8} + \frac{28x^6 (-2 + 6 x)t^2 + 14 x^5 (-2 + 6 x)^4}{(1 - 4 x)^9} +\cdots,
\end{multline*}
where $t=(-3x^2+5x-1)/(1-x)$.

\section{Some results for the case $d=-1$}
In this section we keep analyzing the bivariate generating function given in previous section for the particular case $d=-1$. For this case, 
we provide more detailed results. We denote by $\Q$ the set of all nonempty paths in
$\D_{-1}$ having at least one valley, where the last valley is at ground level. We denote by $\Q_n$  the subset of $\Q$ formed by all paths of
semi-length $n$ and denote by $q_n$ the cardinality of $\Q_n$.  For simplicity, when $d=-1$ (or $e=1$) we use  $L(x,y)$ instead of $L_{1}(x,y)$. 

\begin{theorem}\label{teodnegativo}
 The bivariate generating function $L(x,y)$ is given by
  $$L(x,y)= \frac{(x-1)y\left(1-x(2+y) - \sqrt{(1-x-2xy-2x^2y+x^2y^2-x^3y^2)/(1-x)} \right)}{2(1-2x+x^2-2xy+x^2y)}.$$
\end{theorem}

\begin{proof}
A path $P \in \Q$ can be uniquely decomposed as either
$UD$, $\, UTD$, $\, U\Delta DT,$ or $\, UQDT$, where $\Delta$ is a pyramid, $T \in  \D_{-1}$, and $Q\in \Q$. Therefore,  we obtain the following functional relation
\begin{equation}\label{teodnegativoLxy}
L(x,y)= xy+ xL(x,y) + x\left(\frac{y}{1-x}\right)L(x,y) + xL(x,y)Q(x,y),
\end{equation}
where $$Q(x,y):=\sum_{Q\in \Q}x^{\ell(Q)}y^{\rho(P)}.$$
We are going to obtain an explicit expression for the generating function $Q(x,y)$. Additionally, a path $Q\in\Q$ can be uniquely decomposed  as either
$U\Delta DU\Delta'D$, $\, U\Delta DR$, $\,UR_1 DR_2$, or $\,URDU\Delta D$,
where $\Delta, \Delta'$ are pyramids, and $R, R_1, R_2\in \Q$ (see Figure \ref{Fig4} for a graphical representation of this decomposition).

\begin{figure} [htbp]
\begin{center}
\includegraphics[scale=0.5]{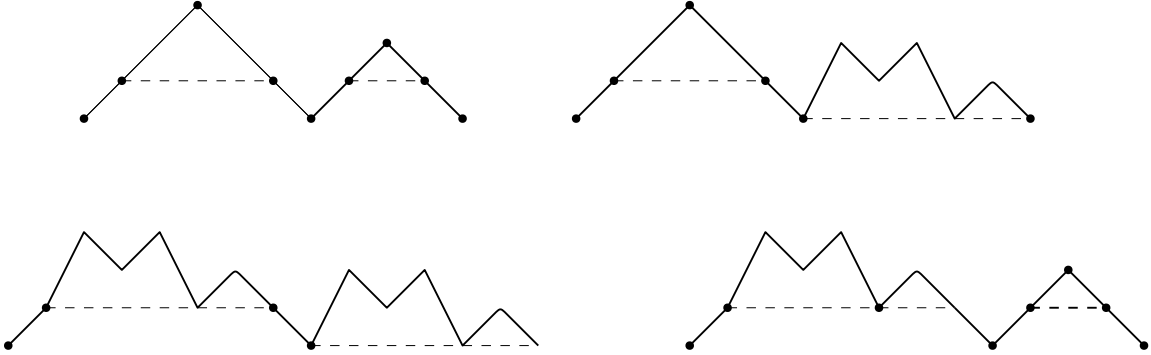}
\end{center}
\caption{Decomposition of a $(-1)$-Dyck path in $\Q$.} \label{Fig4}
\end{figure}

Using the symbolic method, we obtain the functional equation
\begin{align*}
  Q(x, y)&= x^2\left(\frac{y}{1-x}\right)^2+ x\left(\frac{y}{1-x}\right)Q(x, y) + x(Q(x,y))^2 + x^2\left(\frac{y}{1-x}\right)Q(x,y).
\end{align*}
Solving the equation above for $Q(x,y)$, we find that
\begin{align}\label{gFunQ}
Q(x, y)=\frac{1-x-xy-x^2y -\sqrt{(1-x)(1 - x - 2 x y -2 x^2 y + x^2y^2 - x^3y^2)}}{2(1-x)x}.
\end{align}
This and solving \eqref{teodnegativoLxy} for $L(x,y)$ imply the desired result.
\end{proof}

Expressing $L(x,y)$ as a series expansion we obtain these first few terms.
\begin{multline*}
L(x,y)=x y+x^2 \left(y^2+y\right)+x^3 \left(y^3+3 y^2+y\right)+x^4 \left(y^4+\textbf{6}y^3+6 y^2+y\right)\\
+x^5 \left(y^5+10 y^4+19 y^3+10 y^2+y\right)+x^6\left(y^6+15 y^5+46 y^4+45 y^3+15 y^2+y\right)+\cdots
\end{multline*}

Figure \ref{Eje1} depicts all six paths in $\D_{-1}(4)$ with exactly 3 peaks. Notice that this is the bold coefficient of $x^4y^3$ in the above series.

\begin{figure} [htbp]
\begin{center}
\includegraphics[scale=0.8]{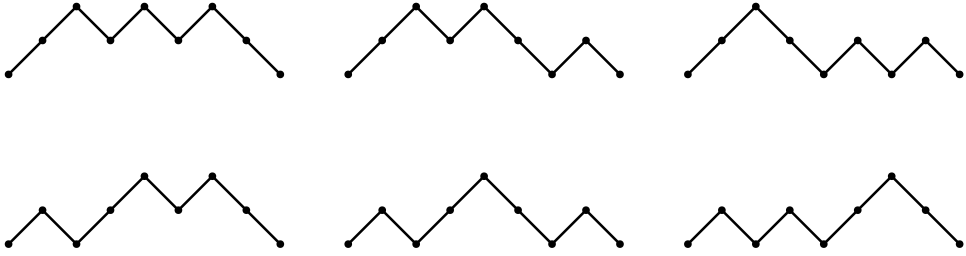}
\end{center}
\caption{All six paths in $\D_{-1}(4)$ with exactly 3 peaks.} \label{Eje1}
\end{figure}

The generating function for the $(-1)$-Dyck paths is given by
\begin{equation}\label{LxForD1}
L(x):=L(x,1)=\frac{-1 + 4 x - 3 x^2 + \sqrt{1 - 4 x + 2 x^2 + x^4}}{2 (1 - 4 x + 2 x^2)}.
\end{equation}
Thus,
\begin{equation*}
L(x):=x+2 x^2+5 x^3+14 x^4+41 x^5+123 x^6+375 x^7+1157 x^8+\cdots.
\end{equation*}

For simplicity for the remaining part of the paper, if there is not ambiguity, we use $r(n)$ instead of $r_{-1}(n)$. Our interest here is to give a combinatorial expression for this sequence.
First of all, we give some preliminary results. Let $b(n)$ be the number of $(-1)$-Dyck paths of semi-length  $n$ that either have no valleys or the last valley is at ground level. Note that $b(n)-1$ is the $n$-th coefficient of the generating function $Q(x,1)$, see  \eqref{gFunQ}, or equivalently 
\begin{align*}
\sum_{n\geq 0}b(n)x^n&=Q(x,1) + \frac{1}{1-x}=\frac{1-x^2-\sqrt{1-4x+2x^2+x^4}}{2(1-x)x}\\
&=1+x+2 x^2+4 x^3+9 x^4+22 x^5+57 x^6+154 x^7+429 x^8+\cdots.
\end{align*}

This generating function coincides with the generating function of the number of Dyck paths of semi-length $n$  that avoid the subpath $UUDU$. From Proposition 5 of \cite{Sap} and \cite[pp. 10]{barry} we conclude the following proposition.

\begin{proposition}
For all $n\geq 0$ we have
$$b(n)=1+\sum_{j=0}^{\lfloor\frac{n-1}{2}\rfloor}\frac{(-1)^j}{n-j}\binom{n-j}{j}\binom{2n-3j}{n-j+1}=\sum_{k=0}^{\lfloor\frac{n}{2}\rfloor}\sum_{j=0}^{n-k}\binom{n-k}{j}N(j,k),$$
where $N(n,k)=\frac{1}{n}\binom nk \binom{n}{k-1}$ are the Narayana numbers, with $N(0,0)=1$. 
\end{proposition}

\begin{remark}
\label{rem1}
Let $\mathcal{B}(n)=\mathcal{Q}_n\cup \{U^nD^n\}$ denote the set of $(-1)$-Dyck paths having either no valleys or the last valley is at height zero. Denote by $\mathcal{B}_{j,k}(n)$ 
the subset of $\mathcal{B}(n)$ of paths that contain exactly $j$ valleys where $k$ of those valleys are $(-1)$-valleys, i.e., there is a valley to the left of it, no valleys in between, and the  
heights differ by $-1$. Take a path $P\in \mathcal{B}_{j,k}(n)$, then $\rho (P)=j+1$. Decompose the path as $$P=U^a\Delta _{t_1}D^{r_1}U^{s_1}\Delta _{t_2}\cdots D^{s_j}U^bD^b,$$ 
where there are $k$ occurrences of $D\Delta _p{\color{red}{D}}U$. So, there are $k$ {\textcolor{red}{red}} down steps indicating a $(-1)$-valley. Notice then that there are $n-k$ down  
steps that are not labeled red and they belong to pyramids, there is no restriction on those, so they can be represented as compositions of $n-k$ down steps on $j+1$ parts, one per  
peak. They are counted by $\binom{n-k-1}{j+1-1}=\binom{n-k-1}{j}$. This means that, numerically, $N(j,k+1)$ corresponds to the sequence of $(-1)$-Dyck paths of semilength $j+k+1$ 
containing $j$ valleys with $k$ of them being $(-1)$-valleys.
\end{remark}

\begin{theorem} The total number of paths in $\D_{-1}(n)$ is given by
$$r(n)=\sum_{\ell=0}^n\sum_{i=0}^{n-\ell-1}\binom{n-\ell-1}{i}q^{(i)}(\ell),$$
where
$$q^{(i)}(n)=\sum_{n_1+n_2+\cdots+n_i=n}b(n_1)b(n_2)\cdots b(n_i).$$
\end{theorem}
\begin{proof}
Let $\mathcal{Q}^{(i)}(n)$ denote the set of $i-$tuples $(P_1,\ldots ,P_i)$ of paths $P_j\in \mathcal{B} =\bigcup _{n\geq 0}\mathcal{B}(n)$, using the notation in Remark \ref{rem1} above, such that $\ell \left (P_1\right )+\cdots+\ell \left (P_i\right )=n$. It is clear that $\left |\mathcal{Q}^{(i)}(n)\right |=q^{(i)}(n)$. Notice that by definition of $\mathcal{B}$ we allow the empty path, $\lambda$, to be counted. Recall, also, that the number of compositions of $n-\ell$ on $i+1$ positive integer parts, denoted as $\mathcal{C}_{i+1}(n-\ell)$, is given by the binomial coefficient  $\binom{(n-\ell) -1}{(i+1)-1}$.

Consider, then, the function $$\varphi :\bigcup _{i,\ell}\left (\mathcal{Q}^{(i)}(\ell)\times \mathbb{C}_{i+1}(n-\ell)\right )\longrightarrow \mathbb{D}_{-1}(n),$$
defined by $\varphi \left ((P_1,\ldots ,P_i),(C_1,\ldots ,C_{i+1})\right )=P$, where $P$ is the path described as
$P=U^{C_1}MU^{C_2}\ldots$,
where $$M=\begin{cases}
D^{C_1},& \text{ if } P_1=\lambda; \\
P_1,& \text{ if } P_1=\Delta;\\
P_1D,& \text{ otherwise}.
\end{cases}$$
Figure \ref{fig:exBij} shows two examples of how the function $\varphi$ works. 

\begin{figure}[H]
    \centering
    \includegraphics[height=5cm,width=10cm]{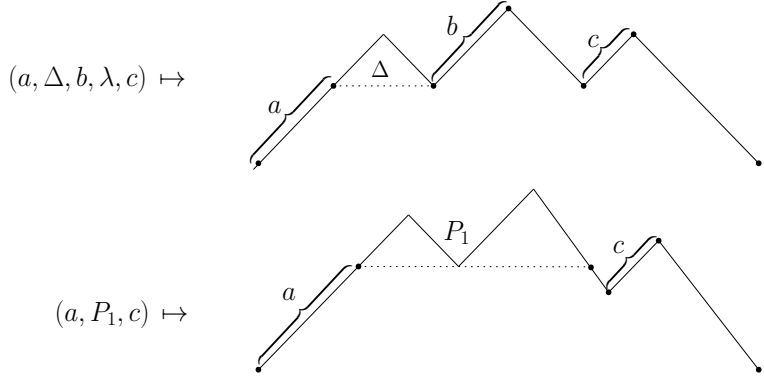}
    \caption{Function $\varphi$ applied to $(a,b,c),(\Delta , \lambda)$ and to $(a,c),(P_1)$.}
    \label{fig:exBij}
\end{figure}
\smallskip

This function is a bijection, where the inverse function is given by decomposing a path by using the following algorithm:
\begin{algorithm}[htbp]
\begin{enumerate}[(1) ]
    \item  If there are $(-1)$-valleys, go to step (2). If there are no consecutive valleys with difference equal to $-1$, then the path is increasing and can be decomposed by using only pyramids $\Delta$ and $\lambda$ in the following way:
    \begin{itemize}
    \item If there are no valleys, then the path is $\Delta _a = U^aD^a$ for some $a\geq 1$, return $(a)$. 
        \item If there is just one valley, the path is $U^aD^bU^cD^d$ for $a\geq b$ and $a+c=b+d$, if $a>b$ return $(a-b,U^bD^b,c)$ and if $a=b$ return $(a,\lambda,c)$.
        \item For two consecutive valleys at the same height (not in the ground), place $(a,\Delta ,b,\lambda)$ and start at the second valley. If they are in the ground, return $(a,\lambda ,b,\lambda)$.
        \item For two consecutive valleys that are not at the same height return $(a,\Delta  _1,b, \Delta _2)$
    \end{itemize}
   
    \item Find the rightmost $(-1)$-valley, i.e., locate the rightmost occurrence of $DU^kD^{k}{\color{red}{D}}U.$ The right part of this string is increasing: go to step (1). Extract the maximal subword that is a Dyck path and call this path $P_i,$ increase the value of $i$ and go to step (1) with the left part of the path.
\end{enumerate}
\caption{Inverse function $\phi$ (or reverse)}
\label{ReverseFunction}
\end{algorithm}

For example, consider the path given in Figure \ref{fig:expro}. First one locates the rightmost $(-1)$-valley (all denoted by a red circle around them) and to the right then the first step says that it is $(1)$. We take out the path $P_1$ and we locate the next $(-1)$-valley, the right part corresponds to $(1,\lambda,1,\lambda,1)$ and the left part of $P_2$ corresponds to $(1,UD,1)$, and so the whole path is encoded by $(1,UD,1,P_2,1,\lambda,1,\lambda,1,P_1,1)$.
\end{proof}

\begin{figure}[H]
    \centering
    \includegraphics[scale=0.8]{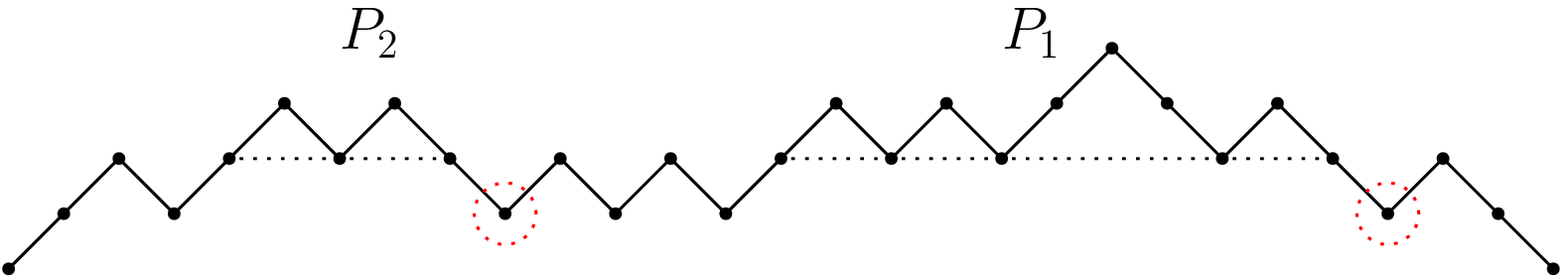}
    \caption{Example inverse function.}
    \label{fig:expro}
\end{figure}

The Corollary \ref{NumberPathQn} is direct consequence of the decomposition given in the proof of Theorem \ref{teodnegativo}. The first result. follows
from Figure \ref{Fig4} and the second result uses the first part and the decomposition $UTD$, $U\Delta DT,$ or $UQDT$ as given in the proof
of Theorem \ref{teodnegativo}.

\begin{corollary}\label{NumberPathQn} If $n>1$, then these hold
\begin{enumerate}
\item If $q_n=|\Q_n|$, then 
 $$q_n=2 q_{n-1}+q_{n-2}+q_{n-3}+\sum _{i=2}^{n-4} q_{i} (q_{n-i-1}-q_{n-i-2})+1,$$
for $n>3$, with the initial values $q_1=0$, $q_2=1$, and $q_3=3$.

\item If $r(n)=|\D_{d}(n)|$, then 
$$r(n)=3 r(n-1)-r(n-2)+q_{n-2}+\sum _{i=2}^{n-3} q_{i} (r(n-i-1)-r(n-i-2)), $$
for $n>3$, with the initial values $r(1)=1$, $r(2)=2$, and $r(3)=5$.
\end{enumerate}
\end{corollary}

The generating function of the sequence $r(n)$ is algebraic of order two, then $r(n)$ satisfies  a recurrence relation with polynomial coefficients.  
This can be automatically solved with Kauers's algorithm \cite{Kau}. In particular we obtain that $r(n)$ satisfies the recurrence relation:

 \begin{multline*}
 2nr(n) - 4nr(n+1) + (12 + 5 n)r(n+2) - 4(15 + 4 n)r(n+3) \\
 + 10(9 + 2 n)r(n+4) - 2(21 + 4 n)r(n+5) + (6 + n)r(n+6)=0, \quad n \geq 6
 \end{multline*}
with the initial values $r(0)=0,  r(1)=1, r(2)=2, r(3)=5, r(4)=14$, and $r(5)=41$.

In Theorem \ref{AsymptoticApproximationr(n)} we give an  asymptotic approximation for the sequence $r(n)$. To accomplish this goal we use the singularity  
analysis method to find the asymptotes of the coefficients of a generating function (see, for example, \cite{flajolet} for the details).

\begin{theorem} \label{AsymptoticApproximationr(n)}
The number of $(-1)$-Dyck paths  has the asymptotic approximation
$$r(n)\sim \frac{\rho^{-n}}{\sqrt{n^3\pi}}\cdot \dfrac{\sqrt{\rho(4-4\rho-4\rho^3)}}{4(-1+4\rho-2\rho^2)}.$$
\end{theorem}

\begin{proof}
The dominant singularity $\rho$ of the generating function $L(x)$ is the smallest real positive root of $1 - 4 x + 2 x^2 + x^4$. From a symbolic computation we find that
$$\rho=\frac{1}{3} \left(-1-\frac{4\ 2^{2/3}}{\sqrt[3]{13+3 \sqrt{33}}}+\sqrt[3]{2 \left(13+3
   \sqrt{33}\right)}\right)\approx 0.295598.$$

From the expression given in \eqref{LxForD1} for $L(x)$ we have
\[
L(x)=\frac{-1+4x-3x^2}{2 (1 - 4 x + 2 x^2)} + \frac{\sqrt{1 - 4 x + 2 x^2 + x^4}}{2 (1 - 4 x + 2 x^2)}
\sim (x-\rho)^{1/2} \frac{\sqrt{\rho(4-4\rho-4\rho^3)}}{2(1-4\rho+2\rho^2)} \quad \text{ as }  x\to \rho.
\]
Therefore,
\begin{equation*}
r(n)\sim \frac{n^{-1/2-1}}{\rho^n(-2\sqrt{\pi})} \frac{\sqrt{\rho(4-4\rho-4\rho^3)}}{2(1-4\rho+2\rho^2)}= \frac{\rho^{-n}}{\sqrt{n^3\pi}}
\frac{\sqrt{\rho(4-4\rho-4\rho^3)}}{4(-1+4\rho-2\rho^2)}. \qedhere
\end{equation*}
\end{proof}

\section{The Area of the  $(-1)$-Dyck paths}

In this section we use generating functions and recursive relations to analyze the distribution of the area of the paths in the set of restricted
$(-1)$-Dyck paths.  We recall that the \emph{area} of  a Dyck path is the sum of the  absolute values of  $y$-components of all points in the path.
We use $\area(P)$ to denote the area
of a path $P$.   From Figure \ref{Example} on Page \pageref{Example},  we can see that $\area(P)=70$.  We use  $a(n)$
to denote the total  area  of all paths in $\D_{-1}(n)$.  In Theorem \ref{GFAreaDDyckPath}  we give a generating function for the sequence $a(n)$.
We now introduce a bivariate  generating function depending on this previous parameter and  $\ell(P)$ (the semi-length of $P$). So,

 $$A(x, q):=\sum_{P\in \D_{-1}}x^{\ell(P)}q^{\area(P)}.$$

 We now give again some terminology needed for the following theorems. Let $\Q \subset \D_{-1}(n)$ be the set formed by all paths having at least one valley, were
the last valley is at ground level; let $\Q_n\subset \Q$ be the set formed by all paths of semi-length $n$, and  let $q_n=|\Q_n|$.

\begin{theorem}\label{GFAreaDDyckPath}
The generating function for the sequence $a(n)$ is given by
\begin{align*}
V(x)&=\sum_{n\geq 0} a(n)x^n=\frac{b(x) - c(x)\sqrt{1-4x+2x^2+x^4}}{(1- x)^2 (1 - 4 x + 2 x^2)^3 (1 - 3 x - x^2 - x^3)},
\end{align*}
where
\begin{align*}
b(x)&=2 x - 23 x^2 + 107 x^3 - 262 x^4 + 359 x^5 - 256 x^6 + 82 x^7 -
 5 x^8 - 10 x^9 + 6 x^{10},\\
c(x)&=x - 10 x^2 + 41 x^3 - 89 x^4 + 108 x^5 - 73 x^6 + 18 x^7 + 2 x^8.
 \end{align*}
\end{theorem}

\begin{proof}
From the decomposition $UD, \, UTD, \, U\Delta DT,$ or $UQDT$ given in the proof of Theorem \ref{teodnegativo} we obtain the functional equation
\begin{align}\label{eqA1}
A(x,q)=xq+xqA(xq^2,q)+ E(x,q)A(x,q) + xqB(xq^2,q)A(x,q),
\end{align}
where $E(x,q):=\sum_{j\geq 1}x^jq^{j^2}$ and $B(x,q):=\sum_{P\in \Q}x^{\ell(P)}q^{\area(P)}$.
Note that  $E(x,q)$ corresponds to the generating function that counts the total number of non-empty pyramids in the given decomposition.

From the decomposition given in Figure \ref{Fig4}, we obtain the functional equation
\begin{align}\label{eqB1}
B(x,q)=E(x,q)^2+E(x,q)B(x,q)+xqB(q^2x,q)B(x,q) + xqB(q^2x,q)E(x,q).
\end{align}
Let $M(x)$ be the generating function of the total area of the $(-1)$-Dyck paths in $\Q$. From the definition of $A(x,q)$ and $B(x,q)$ we have
 $$V(x)=\left.\frac{\partial A(x,q)}{\partial q}\right|_{q=1} \quad \text{and} \quad M(x)=\left.\frac{\partial B(x,q)}{\partial q}\right|_{q=1}.$$
Therefore, differentiating \eqref{eqB1} with respect to $q$ we obtain,
\begin{multline*}
M(x)=\frac{2x^2(1+x)}{(1-x)^4}+\frac{x(x+1)}{(1-x)^3}Q(x) + \frac{x}{1-x}M(x) + xQ(x)^2 + \\
		x\left(M(x) + 2x \frac{\partial Q(x)}{\partial x}\right)\left(Q(x) + \frac{x}{1-x}\right)
				     + xQ(x)\left(M(x) + \frac{x}{1-x}\right) + xQ(x)\frac{x(x+1)}{(1-x)^3},
\end{multline*}
 where $Q(x):=Q(x,1)$ and $Q(x,y)$ is the generating function given in \eqref{gFunQ} on Page \pageref{gFunQ}.

Now, differentiating \eqref{eqA1} with respect to $q$ we obtain,
  \begin{multline}\label{eqA2}
V(x)=x + xL(x) + x\left(V(x) + 2x\frac{\partial L(x)}{\partial x}\right) + \frac{x(x+1)}{(1-x)^3}L(x)\\+\frac{x}{1-x}V(x) + xQ(x)L(x) + x \left.\frac{\partial B(xq^2,q)}{\partial q}\right|_{q=1} L(x) + xQ(x)V(x).
\end{multline}
Solving \eqref{eqB1} for  $B(x,q)$  and substituting into \eqref{eqA2} and then solving the resulting expression for  $V(x)$ we obtain the desired result.
\end{proof}

The first few values of the series of $V(x)$ are
 \begin{align*}
V(x)= \sum_{n\geq 1}a(n)x^n
 &=x + 6 x^2 + 29 x^3 + 130 x^4 + 547 x^5 + 2198 x^6 + 8551 x^7 +
 32508 x^8  + \cdots.
\end{align*}

We recall that for simplicity we use $r(n)$ instead of $r_{-1}(n)$.

\begin{theorem} If $n\ge 1$, then these hold
\begin{enumerate}
\item \label{AreaPart1} If  $A_n$ is the total area of all paths in $\Q_n$, then 
 \begin{multline*}A_n=2 A_{n-1}+A_{n-2}+2 A_{n-3}+q_{n}-q_{n-1}+2n q_{n-2} +2(n-5) q_{n-3}+4 n^2-14 n+13+\\
 \sum _{i=2}^{n-4} 2(A_{i}+i q_{i} +i(i+1)) (q_{n-i-1}-q_{n-i-2}), \quad  n>4,
\end{multline*}
with the initial values $A_1=0$, $A_2=2$, $A_3=13$, and $A_4=58$.

\item \label{AreaPart2} The sequence  $a(n)$ satisfies the recursive relation 
\begin{multline*}
a(n)=3 a(n-1)-a(n-2)+A_{n-2}+2(n-1) q_{n-2}+2 n r(n-1)+2(3-n) r(n-2)\\
-4 r(n-3)+(n-1)^2+ \sum _{i=3}^{n-2} q_{i-1} (a(n-i)-a(n-i-1))\\
+\sum _{i=3}^{n-2} \left(A_{i-1}+(2 i-1) q_{i-1}+i^2\right) (r(n-i)-r(n-i-1)).
\end{multline*}
\end{enumerate}
\end{theorem}
 
\begin{proof}
We prove Part \eqref{AreaPart1}, constructing a recursive relation for the total area  of $\Q_{n}$. This part of the proof is divided in four cases, and we use $P\setminus T$ to denote the subpath resulting after removing the subpath $T$ from the path $P$.

\textbf{Case 1}. We observe that for a fixed $i \in \{1, 2, \dots, n-1\}$ there is exactly one path in $\Q_{n}$ of the form $(XY)^i(XY)^{n-i}$, where its  
area is equal to $i^2 +(n-i)^2$. So, the total area of these types of paths is $\sum_{i=1}^{n-1} (i^2+(n-i)^2)=n(n-1) (2 n-1)/3$.

\textbf{Case 2}. In this case, we find the area of all paths of the form $P_i:=X^iY^iQ$. Note that in $P_1$ the first pyramid is of height one and $Q \in \Q_{n-1}$
and in $P_{n-2}$ the first pyramid is of height $n-2$ and $Q \in \Q_{2}$. These give that all first pyramids $(XY)^i$ run for $i \in \{1,2, \dots, n-2\}$ and
$\Q_{j}$ runs for $j \in \{2, \dots, n-1\}$.

Now from the definition of $P_i$,  we have that for a fixed $i$ there are $q_{n-i}$ paths of form $P_i$ (that is, having a pyramid $(XY)^{i}$ in the beginning of the path).
So, the contribution to the area given by all first pyramids of the form $(XY)^i$, overall paths of the form $P_i$, is equal to $i^2\times q_{n-i}$.
This and the fact that $A_{n-j}$ is the area of  $\Q_{n-j}$, imply that the  total area of all paths of the form $P_i$ is given by $i^2 q_{n-i}+A_{i}$.
Therefore, the total area of these types of paths is $\sum_{i=1}^{n-2} i^2 q_{n-i}+\sum_{j=2}^{n-1}A_{j}$.

\textbf{Case 3}. In this case we find the area of all paths of the form $H_i:=XQ_{\ell}Y(XY)^{i}$ where $Q_{\ell} \in \Q_{n-i-1}$.
Note that similar to the Case 2, the last  pyramids $(XY)^i$ run for $i \in \{1,2, \dots, n-3\}$ and $\Q_{j}$ runs for $j \in \{2, \dots, n-2\}$.
We now observe that for a fixed $i$ there are $q_{n-i-1}$ paths of form $H_i$ (that is, having a pyramid $(XY)^{i}$ in the end of the path).
The contribution to the area given by all last pyramids of the form $(XY)^{i}$, overall paths of the form $H_i$, is equal to $i^2\times q_{n-i-1}$.

We analyze the contribution to the desired area given by $XQ_{\ell}Y= H_{n-i-1}\setminus (XY)^{n-i-1} $ with $Q_{\ell} \in \Q_{i}$. For a fixed $i \in \{2, 3, \dots, n-2\}$ there are $q_{i}$ paths of form $H_{n-i-1}$ having a first subpath of the form $XQ_{\ell}Y$.
Note that $X$ and $Y$ give rise to a trapezoid, where the two  parallel sides have lengths $2i$ and $2i+2$, giving rise to an area of $2i+1$.
So, for a fixed $i$, the contribution to the area given by all first subpaths of the form $XQ_{\ell}Y$ is equal to the area of the trapezoids plus the area of all
paths of the form $Q_{\ell}$  (these are on top of the trapezoids). That is, the area of a trapezoid multiplied by the total number of the paths of the
form $Q_{\ell}$, plus the area of all paths of the form $Q_{\ell}$. Thus, the contribution to the area given by first  subpaths of the form $XQ_{\ell}Y$
(overall paths of the form $H_i$,  for a fixed $i$),  is $((2i+1) \times q_{i} + A_i)$.

We conclude that the total area of these types of paths is $$\sum_{i=1}^{n-3} i^2\times q_{n-i-1}+\sum_{i=2}^{n-2} ((2i+1) \times q_{i} + A_i).$$

\textbf{Case 4}. Finally, we find the area of all paths of the form $T_i:=XQ^{\prime}YQ^{\prime\prime}$ where $Q^{\prime} \in \Q_{i}$ and
$Q^{\prime\prime} \in \Q_{n-i-1}$ for $i \in \{2, 3, \dots, n-3\}$. First of all, we analyze the contribution to the desired area given by all paths of the form
$Q^{\prime\prime} \in \Q_{n-i-1}$ (overall paths of the form $T_i$  for a fixed $i$). Since  $Q^{\prime} \in \Q_{i}$, we know that for a given path
$Q   \in \Q_{n-i-1}$ there are as many paths of the form  $XQ^{\prime}YQ$ as paths in $\Q_{i}$. Thus, for a fixed $i \in \{2, 3, \dots, n-3\}$ we find the
area given by all subpaths $T_i \setminus XQ^{\prime}Y$ for every $Q^{\prime} \in \Q_{i}$. Thus, the area of all subpaths of the form
$Q^{\prime\prime} \in \Q_{n-i-1}$, that is, clearly, equal to $A_{n-i-1} q_{i}$.

We now analyze the contribution to the desired area given by all subpaths of the form $XQ^{\prime}Y$. That is, the area of all subpaths
$T_i\setminus Q^{\prime\prime}$ (overall paths of the form $T_i$  for a fixed $i$). It is easy to see that for a fixed $i \in \{2, 3, \dots, n-3\}$
there are $q_{n-i-1}$  subpaths of the form $XQ^{\prime}Y$.  Note that $X$ and $Y$ give rise to a trapezoid, where the two  parallel sides
have lengths $2i$ and $2i+2$, giving rise to an area of $2i+1$.  So, the contribution to the area given by the first subpaths of the form $XQ^{\prime}Y$  
is equal to the area of the trapezoids plus the area of all  paths of the form $Q^{\prime}$  (these are on top of the trapezoids). Thus, the area of a
trapezoid multiplied by the total number of the  paths of the form $Q^{\prime}$ plus the area of all paths of the form $Q^{\prime}$ and then all of
these multiplied by the total number of paths  of the form $Q^{\prime\prime}$. Thus, the contribution to the area given by the first subpaths of the
form $XQ^{\prime}Y$  (overall paths of the form $T_i$  for a fixed $i$), is  $((2i+1) \times q_{i} q_{n-i-1} + A_i q_{n-i-1})$.

We conclude that the total area of these types of paths is $$\sum_{i=2}^{n-3} A_{n-i-1} q_{i}+\sum_{i=2}^{n-3} ((2i+1) \times q_{i} q_{n-i-1}+A_i q_{n-i-1}).$$

Adding the results from Cases 1-4, we obtain that the recursive relation for the area $A_n$ is given by
\begin{multline*}
A_n=\sum _{i=1}^{n-1} \left(i^2+(n-i)^2\right)+\sum _{i=1}^{n-2} i^2 q_{n-i}+\sum _{i=2}^{n-1} A_{i}+\sum _{i=2}^{n-3} (2 i+1) q_{i} q_{n-(i+1)}+\sum _{i=2}^{n-3} A_{i} q_{n-(i+1)}+\\
\sum _{i=2}^{n-3} A_{i} q_{n-(i+1)}+\sum _{i=2}^{n-2} A_{i}+\sum _{i=1}^{n-3} i^2 q_{n-(i+1)}+\sum _{i=2}^{n-2} (2 i+1) q_{i}.
\end{multline*}
Subtracting $A_n$ from $A_{n+1}$ and simplifying we have
 \begin{multline*}A_n=2 A_{n-1}+A_{n-2}+2 A_{n-3}+(2 n-5) q_{n-3}+(2 n-4) q_{n-2}+q_{n-1}+4 n^2-14 n+15+\\
 \sum _{i=2}^{n-4} (2 A_{i}+(2 i+1) q_{i}) (q_{n-i-1}-q_{n-i-2})+\sum _{i=2}^{n-3} \left(2 i^2-2 i+1\right) (q_{n-i}-q_{n-i-1}).
\end{multline*}
We now rearrange this expression to obtain $q_n$ (see the expression within brackets) given in Corollary \ref{NumberPathQn}

 \begin{multline*} A_n=2 A_{n-1}+A_{n-2}+2 A_{n-3}+(2 n-6) q_{n-3}+(2 n-4) q_{n-2}-q_{n-1}+4 n^2-14 n+13+\\
 \sum _{i=2}^{n-4} 2( A_{i}+ i q_{i}) (q_{n-i-1}-q_{n-i-2})+\sum _{i=2}^{n-3} 2\left( i^2- i\right) (q_{n-i}-q_{n-i-1}) \\
+[2 q_{n-1}+q_{n-2}+q_{n-3} +\sum _{i=2}^{n-4} q_{i} (q_{-i+n-1}-q_{-i+n-2})+1].
\end{multline*}
After some simplifications we obtain the desired recursive relation.

Proof of Part \eqref{AreaPart2}. This part is similar to Part \ref{AreaPart1}. However, in this proof we need to use:
$\D_{-1}(j)$,  $r(i)=|\D_{-1}(i)|$,  $\Q_{j}$, $q_{j}=|\Q_{j}|$, and $A_t$.

\textbf{Case 1}. We find the area of all paths of the form $XQY$, where $Q \in \D_{-1}(n-1)$. Note that $X$ and $Y$ give rise to a trapezoid
of area equal to $2n-1$; this area multiplied by $r(n-1)=|\D_{-1}(n-1)|$ gives that the total area of the trapezoids is $(2n-1)r(n-1)$. The total area of all
paths of the form  $XQY$ is given by the area of all trapezoids and the area of all paths that are on top of the trapezoids. That is, the area of these
types of paths is  $(2n-1)r(n-1)+ a(n-1)$.

\textbf{Case 2}. In this case, we find the area of all paths of the form $K_i:=X^iY^iQ_{\ell}$, where $Q_{\ell}\in \D_{-1}(n-i)$ and $i \in \{1,2, \dots, n-1\}$.
Since $r(n-i)=| \D_{-1}(n-i)|$, we conclude that for a fixed $i$ there are $r(n-i)$ paths of form $K_i$.  So, the contribution to the area given by all first
pyramids of the form $(XY)^i$, overall paths of the form $K_i$, is equal to
$i^2\times r(n-i)$. This and the fact that $a(n-j)$ is the area of  $\Q_{n-j}$, imply that the  total area of all paths of the form $K_i$ is given by
$i^2\times r(n-i)+a(n-i)$.  Therefore, the total area of these typee of paths is $\sum_{i=1}^{n-1} i^2\times r(n-i)+a(n-i)$.

\textbf{Case 3}. Finally, we find the area of all paths of the form $M_i:=XQ^{\prime}YD$ where $Q^{\prime} \in \Q_{i}$ and
$D \in\D_{-1}(n-i-1)$ for $i \in \{2, 3, \dots, n-2\}$. First of all, we analyze the contribution to the desired area given by all paths
$D \in\D_{-1}(n-i-1)$ (overall paths of the form $M_i$  for a fixed $i$).  Since  $Q^{\prime}\in\Q_{i}$, we know that for a given path
$D^{\prime}  \in\D_{-1}(n-i-1)$ there are as many paths of the form  $XQ^{\prime}YD^{\prime}$ as paths in $\Q_{i}$. Thus, for a fixed
$i \in \{2, 3, \dots, n-2\}$ we find the   area given by all subpaths $M_i \setminus XQ^{\prime}Y$ for every $Q^{\prime}\in\Q_{i}$. That is, the area of
all subpaths of $M_i$ of the form $D \in \Q_{n-i-1}$ is equal to $a(n-i-1) q_{i}$.

We now analyze the contribution to the desired area given by all subpaths of the form $XQ^{\prime}Y$ for every $Q^{\prime}\in\Q_{i}$ .
That is, the area of all subpaths  $M_i\setminus D$ (overall paths of the form $M_i$  for a fixed $i$).
It is easy to see that for a fixed $i \in \{2, 3, \dots, n-2\}$ there are $r(n-i-1)$ subpaths of the form $XQ^{\prime}Y$.  Note that $X$ and $Y$ give rise
to a trapezoid, where the two parallel sides have lengths $2i$ and $2i+2$, giving rise to an area of $2i+1$.  So, the contribution to the area given by
the first subpaths of the form $XQ^{\prime}Y$ is equal to the area of the trapezoids plus the area of all paths of the form $Q^{\prime}$  (these are on top of
the trapezoids). Thus, the area of a trapezoid multiplied by the total number of the  paths of the form $Q^{\prime}$ plus the area of all paths of the form
$Q^{\prime}$ and then all of these multiplied by the total number of paths  of the form $D$. Thus, the contribution to the area given by the first subpaths
of the form $XQ^{\prime}Y$  (overall paths of the form $M_i$  for a fixed $i$), is  $((2i+1) \times q_{i} r(n-i-1) + A_i r(n-i-1))$.

We conclude that the total area of these types of paths is $$\sum_{i=2}^{n-2} A_{i} r(n-i-1)+\sum_{i=2}^{n-2} (2i+1) \times q_{i} r(n-i-1).$$

Adding the results from Cases 1-3, we obtain that the recursive relation for the area $a(n)$ is given by

\begin{multline*} a(n)=a(n-1)+(2 n-1) r(n-1)+ \sum _{i=1}^{n-1} i^2 r(n-i)+\sum _{i=1}^{n-1} a(n-i) \\
+\sum _{i=2}^{n-2} q_{i} a(n-i-1)+\sum _{i=2}^{n-2} A_{i} r(n-i-1)+\sum _{i=2}^{n-2} (2i+1) q_{i} r(n-i-1).
\end{multline*}
Subtracting $a(n)$ from $a(n+1)$ and simplifying we have

\begin{multline*}
a(n)=3 a(n-1)-a(n-2)+A_{n-2}+2(n-1) q_{n-2}+(2 n-1) r(n-1)+(3-2 n) r(n-2)+(n-1)^2\\
+\sum _{i=3}^{n-2} q_{i-1} (a(n-i)-a(n-i-1))+\sum _{i=3}^{n-2} A_{i-1} (r(n-i)-r(n-i-1))\\
+\sum _{i=3}^{n-2} (2 i-1) q_{i-1} (r(n-i)-r(n-i-1))+\sum _{i=1}^{n-2} i^2 (r(n-i)-r(n-i-1)).
\end{multline*}

After some other simplifications we have that
\begin{multline*}
a(n)=3 a(n-1)-a(n-2)+A_{n-2}+2(n-1) q_{n-2}+2 n r(n-1) \\ 
+2(3-n) r(n-2)-4 r(n-3)+(n-1)^2+\sum _{i=3}^{n-2} q_{i-1} (a(n-i)-a(n-i-1))\\
+\sum _{i=3}^{n-2} \left(A_{i-1}+(2 i-1) q_{i-1}+i^2\right) (r(n-i)-r(n-i-1)).
\end{multline*}
This completes the proof.
\end{proof}

Notice that the total area of the Dyck paths (cf. \cite{Woan}) is given by $4^n-\binom{2n+1}{n}$.

\section{Acknowledgments}

The first author was partially supported by The Citadel Foundation.
The third author was partially supported by Universidad Nacional de Colombia.

\end{document}